\documentclass[10pt,a4paper]{amsart}
\usepackage[latin1]{inputenc}
\usepackage[T1]{fontenc}
\usepackage[english]{babel}
\usepackage[active]{srcltx}

\usepackage{amsmath,amssymb,amsthm,amsfonts}

\theoremstyle{plain}
\newtheorem{theorem}{Theorem}
\newtheorem{corollary}[theorem]{Corollary}

\DeclareMathOperator{\e}{e}
\newcommand{\C}{\mathbb{C}}

\newcommand{\N}{\mathbb{N}}

 \DeclareMathOperator{\re}{Re}
 \DeclareMathOperator{\im}{Im}
 \DeclareMathOperator{\sign}{sign}

 \renewcommand{\leq}{\leqslant}
\renewcommand{\geq}{\geqslant}

\begin{document}

\title[On the numerical index of real $L_p(\mu)$-spaces]
{On the numerical index of real $\boldsymbol{L_p(\mu)}$-spaces}

\author[Mart\'{\i}n]{Miguel Mart\'{\i}n}
\author[Mer\'{\i}]{Javier Mer\'{\i}}

\address[Mart\'{\i}n \& Mer\'{\i}]{Departamento de An\'{a}lisis Matem\'{a}tico\\
Facultad de Ciencias\\
Universidad de Granada\\
E-18071 - Granada (SPAIN)}

\email{\texttt{mmartins@ugr.es}, \texttt{jmeri@ugr.es}}

\author[Popov]{Mikhail Popov}

\address[Popov]{Department of Mathematics\\
Chernivtsi National University\\
str.~Kotsyubyns'koho 2, Chernivtsi, 58012 (Ukraine)}

\email{\texttt{misham.popov@gmail.com}}


\date{March 10th, 2009. Revised September 28th, 2009. Correction November 26th, 2009.}

\thanks{First and second authors partially supported by Spanish
MEC and FEDER project no.\ MTM2006-04837 and Junta de Andaluc\'{\i}a
and FEDER grants FQM-185 and P06-FQM-01438. Third author supported
by Junta de Andaluc\'{\i}a and FEDER grant P06-FQM-01438 and by Ukr.\
Derzh.\ Tema N 0103Y001103.}

\begin{abstract}
We give a lower bound for the numerical index of the real space
$L_p(\mu)$ showing, in particular, that it is non-zero for
$p\neq 2$. In other words, it is shown that for every bounded
linear operator $T$ on the real space $L_p(\mu)$, one has
$$
\sup\left\{\Bigl|\int |x|^{p-1}\sign(x)\,T x\ d\mu \Bigr|\ :
\ x\in L_p(\mu),\,\|x\|=1\right\} \geq \frac{M_p}{12\e}\|T\|
$$
where $\displaystyle
M_p=\max_{t\in[0,1]}\frac{|t^{p-1}-t|}{1+t^p}>0$ for every
$p\neq 2$. It is also shown that for every bounded linear
operator $T$ on the real space $L_p(\mu)$, one has
$$
\sup\left\{\int |x|^{p-1}|Tx|\ d\mu \ : \ x\in L_p(\mu),\,\|x\|=1\right\}
\geq \frac{1}{2\e}\|T\|.
$$
\end{abstract}

\maketitle

\section{Introduction}
The numerical index of a Banach space is a constant introduced
by G.~Lumer in 1968 (see \cite{D-Mc-P-W}) which relates the
norm and the numerical radius of (bounded linear) operators on
the space. Let us start by recalling the relevant definitions.
Given a Banach space $X$, we will write $X^*$ for its
topological dual and $\mathcal{L}(X)$ for the Banach algebra of
all (bounded linear) operators on $X$. For an operator $T\in
\mathcal{L}(X)$, its \emph{numerical radius} is defined as
$$
v(T):=\sup\{|x^*(Tx)| \ : \ x^*\in X^*,\ x\in X,\ \|x^*\|=\|x\|=x^*(x)=1 \},
$$
and it is clear that $v$ is a seminorm on $\mathcal{L}(X)$
smaller than the operator norm. The \emph{numerical index} of
$X$ is the constant given by
$$
n(X):=\inf\{v(T) \ : \ T\in \mathcal{L}(X),\ \|T\|=1\}
$$
or, equivalently, $n(X)$ is the greatest constant $k\geq 0$ such
that $k\,\|T\| \leq v(T)$ for every $T\in \mathcal{L}(X)$.
Classical references here are the aforementioned paper
\cite{D-Mc-P-W} and the monographs by F.~Bonsall and J.~Duncan
\cite{B-D1,B-D2} from the seventies. The reader will find the
state-of-the-art on the subject in the recent survey paper
\cite{KaMaPa} and references therein. We refer to all these
references for background.

Let us comment on some results regarding the numerical index which
will be relevant in the sequel. First, it is clear that $0\leq
n(X)\leq 1$ for every Banach space $X$, and $n(X)>0$ means that the
numerical radius and the operator norm are equivalent on
$\mathcal{L}(X)$. In the real case, all values in $[0,1]$ are
possible for the numerical index. In the complex case one has
$1/\e\leq n(X)\leq 1$ and all of these values are possible. Let us
also mention that $n(X^*)\leq n(X)$, and that the equality does not
always hold. Anyhow, when $X$ is a reflexive space, one clearly gets
$n(X)=n(X^*)$. Second, there are some classical Banach spaces for
which the numerical index has been calculated. For instance, the
numerical index of $L_1(\mu)$ is $1$, and this property is shared by
any of its isometric preduals. In particular, $n\bigl(C(K)\bigr)=1$
for every compact $K$ and $n(Y)=1$ for every finite-codimensional
subspace $Y$ of $C[0,1]$. If $H$ is a Hilbert space of dimension
greater than one then $n(H)=0$ in the real case and $n(H)=1/2$ in
the complex case.

Let $(\Omega, \Sigma, \mu)$ be a measure space and $1< p<\infty$. We
write $L_p(\mu)$ for the real or complex Banach space of measurable
scalar functions $x$ defined on $\Omega$ such that
$$
\|x\|_p:=\left(\int_\Omega|x|^p\,d\mu \right)^{\frac{1}{p}}<\infty.
$$
We use the notation $\ell_p^m$ for the $m$-dimensional
$L_p$-space. For $A \in \Sigma$, $\chi_{A}$ denotes the
characteristic function of the set $A$. We write $q=p/(p-1)$
for the conjugate exponent to $p$ and
$$
M_p:=\max_{t\in[0,1]} \frac{|t^{p-1}-t|}{1+t^p}=\max_{t\geq 1} \frac{|t^{p-1}-t|}{1+t^p}\, ,
$$
(which is the numerical radius of the operator $T(x,y)=(-y,x)$
defined on the real space $\ell_p^2$, see
\cite[Lemma~2]{MarMer-LP} for instance).

The problem of computing the numerical index of the
$L_p$-spaces was posed for the first time in the seminal paper
\cite[p.~488]{D-Mc-P-W}. There it is proved that
$\bigl\{n(\ell_p^{2})\ : \ 1< p < \infty\bigr\}=[0,1[$ in the
real case, even though the exact computation of $n(\ell_p^2)$
is not achieved for $p\neq 2$ (even now!). Recently, some
results have been obtained on the numerical index of the
$L_p$-spaces \cite{Eddari,Eddari2,Eddari3,MarMer-LP,M-P}.
\begin{itemize}
\item[(a)] The sequence $\bigl(n(\ell_p^m)\bigr)_{m\in\N}$
    is decreasing.
\item[(b)] $n\bigl(L_p(\mu)\bigr)=\inf \{n(\ell_p^m)\,:\,
    m\in\N\}$ for every measure $\mu$ such that
    $\dim\bigl(L_p(\mu)\bigr)=\infty$.
\item[(c)] In the real case, $\displaystyle
    \max\left\{\frac{1}{2^{1/p}},\
    \frac{1}{2^{1/q}}\right\}\,M_p\leq n(\ell_p^{2})\leq
    M_p$.
\item[(d)] In the real case, $n(\ell_p^m)>0$ for $p\neq 2$
    and $m\in \N$.
\end{itemize}
The aim of this paper is to give a lower estimation for the
numerical index of the \emph{real} $L_p$-spaces. Concretely, it
is proved that
\begin{equation}\label{eq:final}
n\bigl(L_p(\mu)\bigr)\geq \frac{M_p}{12\e}\,.
\end{equation}
As $M_p>0$ for $p\neq 2$, this extends item (d) for
infinite-dimensional real $L_p$-spaces, meaning that the
numerical radius and the operator norm are equivalent on
$\mathcal{L}\bigl(L_p(\mu)\bigr)$ for every $p\neq 2$ and every
positive measure $\mu$. This answers in the positive a question
raised by C.~Finet and D.~Li (see \cite{Eddari2,Eddari3}) also
posed in \cite[Problem~1]{KaMaPa}.

The key idea to get this result is to define a new seminorm on
$\mathcal{L}\bigl(L_p(\mu)\bigr)$ which is in between the
numerical radius and the operator norm, and to get constants of
equivalence between these three seminorms. Let us give the
corresponding definitions.

For any $x \in L_p(\mu)$, we denote
$$
x^\# = \begin{cases} |x|^{p-1} \sign(x) & \text{ in the real case}, \\
|x|^{p-1} \sign(\overline{x}) & \text{ in the complex case}, \end{cases}
$$
which is the unique element in $L_q(\mu)$ such that
$$
\|x\|_p^p=\|x^\#\|_q^q \qquad \text{and} \qquad \int_\Omega x\,x^\#\ d\mu =
\|x\|_p\,\|x^\#\|_q=\|x\|_p^p.
$$
With this notation, for $T\in \mathcal{L}\bigl(L_p(\mu)\bigr)$ one
has
\begin{align*}
v(T)&=\sup\left\{\Bigl|\int_\Omega x^\# T x\,d\mu\Bigr|\ : \ x\in L_p(\mu),\
 \|x\|_p=1\right\}.
\end{align*}
Here is our new definition. Given an operator $T \in \mathcal{L}
\bigl(L_p(\mu)\bigr)$, the \emph{absolute numerical radius} of $T$
is given by
\begin{align*}
|v|(T) &:= \sup \, \left\{ \int_\Omega |x^\# T x | \, d \mu\ :\
x \in L_p(\mu),\, \|x\|_p=1 \right\} \\ &= \sup \,
\left\{ \int_\Omega |x|^{p-1} |Tx| \, d \mu\ :\ x \in L_p(\mu),\, \|x\|_p=1 \right\}
\end{align*}
Obviously,
\begin{equation*}
v(T) \leq |v|(T)\leq \|T\| \qquad \bigl(T \in
\mathcal{L}\bigl(L_p(\mu)\bigr)\,\bigl).
\end{equation*}
Given an operator $T$ on the \emph{real} space $L_p(\mu)$, we will
show that
$$
v(T) \geq \, \frac{M_p}{6} \, |v|(T)
\qquad \text{and} \qquad
|v|(T) \geq \, \frac{n\bigl(L_p^\mathbb C(\mu)\bigr)}{2} \,\|T\|\, ,
$$
where $n\bigl(L_p^\mathbb C(\mu)\bigr)$ is the numerical index
of the \emph{complex} space $L_p(\mu)$. Since
$n\bigl(L_p^\mathbb C(\mu)\bigr)\geq 1/\e$ (as for any complex
space, see \cite[Theorem~4.1]{B-D1}), the above two
inequalities together give, in particular, the inequality
\eqref{eq:final}.

\section{The results}

We start proving that the numerical radius is bounded from
below by some multiple of the absolute numerical radius.

\begin{theorem}
Let $1<p<\infty$ and let $\mu$ be a positive meassure. Then, every
bounded linear operator $T$ on the real space $L_p(\mu)$ satisfies
$$
v(T) \geq \, \frac{M_p}{6} \, |v|(T),
$$
where $\displaystyle M_p=\max_{t\geq 1}
\frac{|t^{p-1}-t|}{1+t^p}$.
\end{theorem}

\begin{proof}
Since $|v|$ is a seminorm, we may and do assume that $\|T\| = 1$.
Suppose that $|v|(T) > 0$ (otherwise there is nothing to prove),
fix any $0<\varepsilon < |v|(T)$ and choose $x \in L_p(\mu)$ with
$\|x\|=1$ such that
$$
\int_\Omega |x^\# Tx| \, d \mu \geq |v|(T) - \varepsilon
\stackrel{def}{=} 2 \beta_0 > 0.
$$
Now, set $A = \{ t \in \Omega\,: \ x^\#(t) [Tx](t) \geq 0 \}$
and $B = \Omega \setminus A$. Then
\begin{equation*}
\int_A x^\# Tx \, d \mu - \int_B x^\# Tx \, d \mu =
\int_\Omega|x^\# Tx | \, d \mu \geq 2 \beta_0
\end{equation*}
and so at least one of the summands above is greater than or
equal to $\beta_0$. Without loss of generality, we assume that
$$
\beta \stackrel{def}{=} \int_A x^\# Tx \, d \mu \geq
\beta_0
$$
(otherwise we consider $-T$ instead of $T$). Remark that
\begin{equation} \label{assump1}
\Bigl| \int_\Omega x^\# Tx \, d \mu \Bigr| \leq v(T)
\qquad \text{and} \qquad \Bigl| \int_B x^\# T (x
\chi_{B}) \, d \mu \Bigr| \leq \left\|(x\chi_B)^\#\right\|_q \|x\chi_B\|_p\, v(T)\leq v(T).
\end{equation}
Now, put $y_\lambda = x + \lambda x \chi_{B}$ for each $\lambda
\in [-1, \infty)$. Observe that
\begin{equation} \label{assump2}
\|y_\lambda^\#\|_q \|y_\lambda\|_p = \|y_\lambda\|_p^p = \int_A
|x|^p \, d \mu + (1 + \lambda)^p \int_B |x|^p \, d \mu \leq
\max \bigl\{ 1, (1+\lambda)^p \bigr\},
\end{equation}
which obviously implies that
\begin{equation}\label{eq:thm-v-w-ylambda}
\Big|\int_\Omega y_\lambda^\#Ty_\lambda\, d\mu\Big|\leq v(T)
\left\|y^\#_\lambda\right\|_q\|y_\lambda\|_p\leq
v(T)\max\bigl\{1,(1+\lambda)^p\bigr\}.
\end{equation}
On the other hand, using that
$y_\lambda^\#=x^\#\chi_{A}+(1+\lambda)^{p-1}x^\#\chi_{B}$ and
\eqref{assump1}, we deduce that
\begin{align*}
\Big|\int_\Omega y_\lambda^\#Ty_\lambda\, d\mu\Big| & =
\Big| \beta + \lambda\int_A x^\#T(x\chi_{B})\,d\mu + (1 +
\lambda)^{p-1} \int_Bx^\#Tx\, d\mu + \lambda
(1+\lambda)^{p-1} \int_B x^\#T(x\chi_{B})\,d\mu\Big| \\ & \geq
 \Big| \beta + \lambda \int_A x^\#T(x\chi_{B})\, d \mu - (1 +
\lambda)^{p-1} \beta \Big| \\ & \qquad \  - (1+\lambda)^{p-1}
\Big|\int_\Omega x^\#Tx\, d\mu \Big|-|\lambda|
(1+\lambda)^{p-1}\Big| \int_Bx^\#T(x\chi_{B})\,d\mu \Big|
\\ & \geq \Big| \bigl(1-(1+\lambda)^{p-1}\bigr) \beta
+\lambda\int_Ax^\#T(x\chi_{B})\,d\mu\Big| -
\bigl(1+|\lambda|\bigr)(1+\lambda)^{p-1} v(T).
\end{align*}
This, together with \eqref{eq:thm-v-w-ylambda}, gives us that
\begin{equation} \label{lskf}
v(T)\Big((1+|\lambda|)(1+\lambda)^{p-1}+\max\{1,(1+\lambda)^{p}\}\Big)
\geq
\\
\Big| (1-(1+\lambda)^{p-1})\beta +
\lambda\int_Ax^\#T(x\chi_{B})d\mu\Big|.
\end{equation}
Therefore, putting $\displaystyle
a=\beta^{-1}\int_Ax^\#T(x\chi_{B})\,d\mu$ and
$$
f(\lambda)=|\lambda|^{-1}
\Bigl(\bigl(1+|\lambda|\bigr)(1+\lambda)^{p-1}+\max
\bigl\{1,(1+\lambda)^{p} \bigr\}\Bigr) \qquad \bigl(\lambda
\in [-1, \infty)\setminus\{0\}\bigr),
$$
and multiplying \eqref{lskf} by $|\lambda|^{-1}\beta^{-1}$, we
obtain that
$$
\beta^{-1}v(T)f(\lambda) \geq
\left|\frac{1-(1+\lambda)^{p-1}}{\lambda}-a\right|
$$
for every $\lambda\in[-1,\infty)\setminus\{0\}$. Thus,
\begin{align*}
\beta^{-1}v(T) \bigl(1 + f(\lambda) \bigr) & = \beta^{-1}v(T) \bigl(
f(-1) + f(\lambda) \bigr) \\ & \geq
\bigl|-1-a \bigr| +
\left|\frac{1-(1+\lambda)^{p-1}}{\lambda}-a\right| \geq
\left|\frac{(1+\lambda)^{p-1}-1}{\lambda}-1\right|
\end{align*}
for every $\lambda\in[-1,\infty)\setminus\{0\}$ or, equivalently,
$$
v(T) \geq \,\beta\, \frac{\bigl| (1 + \lambda)^{p-1} - 1 - \lambda
\bigr|}{ | \lambda| +
\bigl(1+|\lambda|\bigr)(1+\lambda)^{p-1}+\max
\bigl\{1,(1+\lambda)^{p} \bigr\}}
$$
for every $\lambda\in[-1,\infty)$. Now we restrict ourselves to
$\lambda \geq 0$ and setting $t = 1+ \lambda$, we obtain that
\begin{equation*} 
v(T) \geq \,\beta\,\frac{|t^{p-1} - t|}{ t -
1+2t^p} = \,\beta\, \frac{|t^{p-1} - t|}{1+t^p} \, \frac{1+t^p}{t-1+2t^p}
\end{equation*}
for every $t \in [1, \infty)$. Since it obviously holds that
$$
\frac{1+t^p}{t-1+2t^p}\geq\frac13
$$
for each $t \in [1, \infty)$, one obtains
that
$$
v(T) \geq \,  \frac{\beta}{3} \, \sup\limits_{t\geq 1} \frac{|t^{p-1}
- t|}{1+t^p} \geq \frac{|v|(T) - \varepsilon}{6} \, \sup\limits_{t
\geq 1} \frac{|t^{p-1} - t|}{1+t^p} = \frac{|v|(T) -
\varepsilon}{6} \, M_p,
$$
which is enough in view of the arbitrariness of $\varepsilon$.
\end{proof}

Our next goal is to prove an inequality relating the absolute
numerical radius and the norm of operators on real
$L_p$-spaces.

\begin{theorem}
Let $1<p<\infty$ and let $\mu$ be a positive measure. Then, every
bounded linear operator $T$ on the real space $L_p(\mu)$ satisfies
$$
|v|(T) \geq \,\frac{n\bigl(L_p^\mathbb C(\mu)\bigr)}{2} \, \|T\|,
$$
where $n\bigl(L_p^\mathbb C(\mu)\bigr)$ is the numerical index
of the \emph{complex} space $L_p(\mu)$.
\end{theorem}

\begin{proof}
We consider the complex linear operator $T_\mathbb C \in
\mathcal{L} \bigl(L_p^\C(\mu)\bigr)$ given by
\begin{equation} \label{eq:pppq}
T_{\mathbb C} (x) = T ( \re x) + i \, T ( \im x)\qquad \bigl(x \in L_p^\mathbb C(\mu)\bigr).
\end{equation}
Evidently, $\|T\| \leq \|T_\mathbb C\|$. Now, consider any simple
function $ x = \sum\limits_{j=1}^m a_j \e^{i \theta_j} \chi_{{A_j}}
\in L_p^\mathbb C(\mu)$ where $m \in \mathbb N$,  $a_j \geq 0$,
$\theta_j \in [0, 2 \pi)$, the sets $A_1,\dots, A_m\in\Sigma$ are
pairwise disjoint, $\sum\limits_{j=1}^m a_j^p\mu(A_j)=1$, and
observe that $x^\#\in L_q^\C(\mu)$ is given by the formula
$$
x^\# = \sum\limits_{j=1}^m a_j^{p-1} \e^{-i \theta_j}
\chi_{{A_j}}.
$$
Then, writing
$$
\alpha_{j,k} = \int_{A_j} T_{\mathbb C}(\chi_{A_k})\, d \mu
=\int_{A_j} T(\chi_{A_k})\, d\mu,
$$
we obtain that
\begin{align}\label{eq:thm-long}
\Bigl| \int_{\Omega} x^\# T_{\mathbb C}(x) \, d \mu \Bigr| &=
\Bigl| \sum\limits_{j=1}^m a_j^{p-1} \e^{-i \theta_j}
\sum\limits_{k=1}^m a_k \e^{i \theta_k} \alpha_{j,k} \Bigr| \leq
\sum\limits_{j=1}^m a_j^{p-1} \Bigl| \sum\limits_{k=1}^m a_k \e^{i
\theta_k}
\alpha_{j,k} \Bigr|  \notag \\
&\leq \sum\limits_{j=1}^m a_j^{p-1} \left( \Bigl| \sum\limits_{k=1}^m
a_k \cos(\theta_k)\, \alpha_{j,k} \Bigr| + \Bigl| \sum\limits_{k=1}^m
a_k \sin(\theta_k)\, \alpha_{j,k} \Bigr| \right) \\
&\leq 2 \, \max\limits_{(z_k) \in [-1,1]^m} \sum\limits_{j=1}^m
a_j^{p-1} \Bigl| \sum\limits_{k=1}^m a_k z_k \alpha_{j,k} \Bigr| =
2 \, \max\limits_{(z_k) \in \{-1,1\}^m} \sum\limits_{j=1}^m
a_j^{p-1} \Bigl| \sum\limits_{k=1}^m a_k z_k \alpha_{j,k}
\Bigr|,\notag
\end{align}
where the last equality follows from the convexity of the
function $f: [-1,1]^m \longrightarrow \mathbb R$ defined by
$$
f(z_1, \ldots, z_m) = \sum\limits_{j=1}^m a_j^{p-1} \Bigl|
\sum\limits_{k=1}^m a_k z_k \alpha_{j,k} \Bigr|.
$$
On the other hand, for any finite sequence $(z_k) \in
\{-1,1\}^m$, putting
$$
y_{(z_k)} = \sum\limits_{j=1}^m a_j z_j \chi_{A_j} \in
L_p(\mu),
$$
one has $\|y_{(z_k)}\|=1$ and that
\begin{align*}
\int_{\Omega} \bigl| y_{(z_\ell)}^\# T(y_{(z_\ell)}) \bigr| \, d \mu
& = \int_{\Omega} \Bigl|
\sum\limits_{j=1}^m a_j^{p-1} z_j\, \chi_{A_j} \sum\limits_{k=1}^m
a_k z_k T (\chi_{A_k}) \Bigr| \, d \mu \\ & = \sum\limits_{j=1}^m
\int_{A_j} \Bigl| a_j^{p-1} z_j \sum\limits_{k=1}^m
a_k z_k T (\chi_{A_k}) \Bigr| \, d \mu \\ &
= \sum\limits_{j=1}^m a_j^{p-1} \int_{A_j} \Bigl|
\sum\limits_{k=1}^m a_k z_k T \chi_{A_k} \Bigr| \,
d \mu \\ & \geq \sum\limits_{j=1}^m a_j^{p-1}  \Bigl| \int_{A_j}
\sum\limits_{k=1}^m a_k z_k T \chi_{A_k} \, d
\mu\Bigr| = \sum\limits_{j=1}^m a_j^{p-1}  \left| \sum\limits_{k=1}^m a_k
z_k \alpha_{j,k} \right|.
\end{align*}
This, together with \eqref{eq:thm-long}, implies that
\begin{align*}
2 |v|(T) & \geq 2 \, \max\limits_{z_\ell \in \{-1,1\}} \int_{\Omega}
\bigl| y_{(z_\ell)}^\# T (y_{(z_\ell)}) \bigr| \, d \mu \\ &\geq 2
\, \max\limits_{z_\ell \in \{-1,1\}} \sum\limits_{j=1}^m
a_j^{p-1}  \Bigl| \sum\limits_{k=1}^m a_k z_k \alpha_{j,k}
\Bigr| \geq \Bigl| \int_{\Omega} x^\# T_{\mathbb C}(x) \, d \mu
\Bigr|.
\end{align*}
Since the set of all simple functions is dense in $L_p^\mathbb
C(\mu)$, it follows from \cite[Theorem~9.3]{B-D1} that the
above inequality implies that
\begin{equation*}
2 |v|(T) \geq v(T_\C) \geq n\bigl(L_p^\mathbb C(\mu)\bigr) \|T_{\mathbb C}\|
\geq n\bigl(L_p^\mathbb C(\mu)\bigr) \|T\|.\qedhere
\end{equation*}
\end{proof}

It remains to notice that $n\bigl(L_p^\mathbb C(\mu)\bigr) \geq
1/\e$ (as happens for any complex Banach space, see
\cite[Theorem~4.1]{B-D1}), to get the following consequence
from the above two theorems.

\begin{corollary}
Let $1<p<\infty$ and let $\mu$ be a positive measure. Then, in the
real case, one has
$$
n\bigl(L_p(\mu)\bigr)\geq \frac{M_p}{12\e}
$$
where $\displaystyle M_p=\max_{t\geq 1}
\frac{|t^{p-1}-t|}{1+t^p}$.
\end{corollary}

Since, clearly, $M_p>0$ for $p\neq 2$, we get the following
consequence which answers in the positive a question raised by
C.~Finet and D.~Li (see \cite{Eddari2,Eddari3}) also posed in
\cite[Problem~1]{KaMaPa}.

\begin{corollary}
Let $1<p<\infty$, $p\neq 2$ and let $\mu$ be a positive measure.
Then $n\bigl(L_p(\mu)\bigr)>0$ in the real case. In other words, the
numerical radius and the operator norm are equivalent on
$\mathcal{L}\bigl(L_p(\mu)\bigr)$.
\end{corollary}

It is a particular case of \cite[Theorem~2.2]{Eddari3} that
$n\bigl(L_p(\mu)\bigr)= \inf\limits_m\ n(\ell^m_p)$ for every
infinite-dimensional $L_p(\mu)$-space. For finite-dimensional
spaces, $n(\ell_p^m)\leq n(\ell_p^2)$ for every $m\geq 2$ since
$\ell_p^2$ is an $\ell_p$-summand on $\ell_p^m$ and we may use
\cite[Remark~2.a]{M-P}. On the other hand, it is clear that
$n(\ell_p^2)\leq M_p$ (since $M_p$ is the numerical radius of a
norm-one operator on the real $\ell_p^2$, see
\cite[Lemma~2]{MarMer-LP} for instance). It then follows that
$$
n\bigl(L_p(\mu)\bigr)\leq M_p
$$
for every $1<p<\infty$ and every positive measure $\mu$ such
that $\dim\bigl(L_p(\mu)\bigr)\geq 2$. We do not know whether
the above inequality is actually an equality.

\vspace{1cm}

\noindent\textbf{Acknowledgments:\ } The authors would like to
thank Rafael Pay\'{a} for fruitful conversations concerning the
matter of this paper.


\end{document}